\documentclass[11pt, a4paper]{amsart}

\usepackage[utf8]{inputenc} 

\newtheorem{theorem}{Theorem}[section]
\newtheorem{corollary}[theorem]{Corollary}
\newtheorem{main}{Main Theorem}
\newtheorem{lemma}[theorem]{Lemma}

\newtheorem{proposition}[theorem]{Proposition}

\theoremstyle{definition}
\newtheorem{definition}[theorem]{Definition}
\newtheorem{remark}[theorem]{Remark}
\newtheorem{example}{Example}

\usepackage{graphicx}

\usepackage{physics} 

\title{Ergodicity of dynamical systems without uniqueness of orbits}
\author{Tomoharu Suda}
\date{\today}
\address{Department of Applied Mathematics, 
Tokyo University of Science}
\email{tomoharu.suda@rs.tus.ac.jp}

\begin{document}
\begin{abstract}
Recently, there has been considerable interest in the study of non-deterministic dynamical systems. To analyze the chaotic behavior of such systems from a measure-theoretic viewpoint, it is desirable to consider ergodicity. However, the classical definition of ergodicity involves invariant sets, whose definition is not unique for non-deterministic dynamical systems.  Thus, we are led to the question of which invariance yields an interesting definition of ergodicity. Here, we propose a definition based on the strong backward invariance and show that analogs of classical results hold. We also consider implications of the Birkhoff ergodic theorem for systems without uniqueness of trajectories.
\end{abstract}
\maketitle

\section{Introduction}
Classical dynamical systems theory concerns systems where the evolution of the state depends uniquely on the current state and the uniqueness of orbits holds. However, in applications, there are important examples of systems where this property does not hold, such as mechanical systems with dry friction, ecological models with prey switching, or epidemic models with interventions \cite{bernardo2008piecewise, carvalho2020sliding, carvalho2021global}. As one of the main interests of dynamical systems theory is to study the phenomenon of chaos, where complex behavior emerges from simple systems, the chaotic behavior of such systems without uniqueness of orbits has been considered as well, and examples of such systems are known \cite{buzzi2016chaotic, ANTUNES202352}. To illustrate chaotic behavior observable in systems without uniqueness of trajectories, here we present a simple example given by piecewise continuous vector fields.

\begin{example}
Consider a piecewise continuous vector field $X$ defined on the half disk $D^+ := \{(x,y) \mid x^2+y^2 \leq 0, y \geq 0\}$:
\begin{equation}\label{eqn_hd}
 X(x,y) = \begin{cases}
    (y,-x) & (y>0)\\
   (-1,0) &(y=0)
   \end{cases}
\end{equation}
Solutions or trajectories of this system can be defined by gluing the partial trajectories (Figure \ref{ex_hd}). 
\end{example}
\begin{figure}[htbp]
  \centering
  \includegraphics[width=\linewidth,
    clip ]{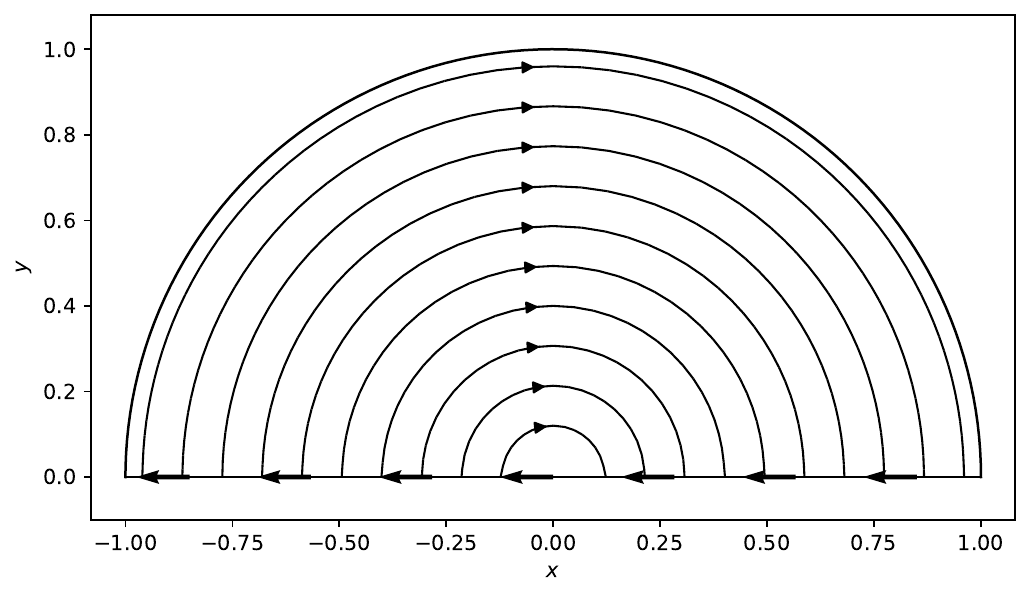}
  \caption{The half-disk system.}
  \label{ex_hd}
\end{figure}

In particular, the theory has been developed considerably on the topological aspects, and natural generalizations of the classical theory have been obtained \cite{buzzi2016chaotic, ANTUNES202352}. Based on these results, for example, it is easy to show that the system (\ref{eqn_hd}) is chaotic in the sense as defined there. Namely, it is topologically transitive and has sensitive dependence on initial points. Further, its periodic orbits are dense.

On the other hand, in the study of classical dynamical systems, the measure-theoretic approach is an essential tool for the analysis of chaotic behavior \cite{walters2000introduction, brin2002introduction}. Its generalizations to systems without uniqueness have been considered \cite{miller1999invariant, moussa2017invariant, novaes, benaim2000ergodic, faure2013ergodic}. Yet a notable obstruction is ergodicity, which plays a central role in the classical setting. Its definition involves an invariant set and an invariant measure, but the former has two flavors in systems without uniqueness. Thus, we are naturally led to the question of how to generalize it.

The purpose of the present article is to consider a generalization of the notion of ergodicity so that analogs of classical results can be obtained for the generalized dynamical systems.
Here, we construct a theory based on the axiomatic theory of ordinary differential equations first described by Yorke and expanded by the author \cite{yorke1969spaces, suda2022equivalence, suda2023equivalence, SUDA20241}. This theory enables us to obtain a concise description of generalized dynamical systems, and a theory of invariant measures has been developed, so it provides us with a framework to consider measure-theoretic aspects \cite{SUDA20241}. Also, we note that this approach is akin to the orbit space theory in spirit \cite{orbitspace}. 

To address non-uniqueness, we define ergodicity in terms of strong invariance. Namely, an invariant measure is said to be ergodic if it assigns measure 0 or 1 to strongly backward invariant subsets (Definition \ref{def_ergodicity}). Under this definition, the system (\ref{eqn_hd}) clearly has an ergodic invariant measure. 

Further, we have generalizations of the classical consequences of ergodicity. For example, we have the following result on the topological property of the support of ergodic invariant measures (cf. Proposition 4.2.2 in \cite{brin2002introduction}).
\begin{main}
If $S$ satisfies the switching axiom and $\mu$ is ergodic, $\mu$-almost every point $x$ has a forward orbit dense in the support of $\mu$. 
\end{main}
Notably, the Birkhoff ergodic theorem can be applied to a trajectory-wise description of the system (see Corollary \ref{cor_birkhoff} in Section \ref{section_birkhoff}). Still, it does not generalize precisely to a set-valued semigroup defined by trajectories, even under an assumption stronger than mere ergodicity.
\begin{main}
Let $S \subset C(\mathbb{R}, X)$ be a compact shift-invariant subset. Let $\mu$ be a finely ergodic invariant measure with an ergodic shift invariant measure $\nu$ such that $\mu = (\pi)_*\nu$. Then, for all measurable subsets $A$ and $B$ of $X$,
\[
 \liminf_{t\to \infty} \frac{1}{t} \int_{0}^t \mu\qty(V_S(s)^{-1}A \cap B) \mathrm{d}s \geq \mu(A) \mu(B).
\]
The inequality can be strict.
\end{main}
We also define mixing and show that weakly mixing systems are ergodic (Definition \ref{def_wmix} and Theorem \ref{thm_wmix}). Thus, the basic ingredients of the ergodic theory can be generalized.

This article is organized as follows. In Section \ref{section_prelim}, we recall basic definitions and prepare preliminary results used throughout the article. In Section \ref{section_ergodic}, we define ergodicity of invariant measures of generalized dynamical systems and consider basic properties of them. In particular, we show that generalizations of classical results hold. In Section \ref{section_birkhoff}, we discuss the aspects related to the Birkhoff ergodic theorem and mixing properties.  In Section \ref{section_conclusion}, we give concluding remarks.

\section{Preliminaries}\label{section_prelim}
In this section, we present preliminary results used in the discussion below.
In what follows, $X$ is assumed to be a compact metric space. The set $C(\mathbb{R}, X)$ of all continuous maps $\mathbb{R} \to X$ is assumed to be given the compact-open topology.

\subsection{Invariance and ergodicity of flows}
First, we recall basic definitions of invariance and ergodicity for flows.
\begin{definition}
Let $\Phi$ be a flow on $X$.
\begin{enumerate}
\item A subset $A \subset X$ is \textbf{backward invariant} if 
\[
 \qty(\Phi^t)^{-1} A \subset A
\]
for all $t\geq 0$. 

\item Also, $A \subset X$ is \textbf{invariant} if 
\[
 \qty(\Phi^t)^{-1} A = A
\]
for all $t\geq 0$. 
\item A measure $\mu$ is \textbf{invariant} if it is a Borel probability measure on $X$ and
\[
 \mu\qty(\qty(\Phi^t)^{-1} A) = \mu(A)
\]
holds for all Borel subset $A \subset X$ and for all $t\geq 0$.
\item An invariant measure $\mu$ is \textbf{ergodic} if $\mu(A) >0$ implies $\mu(A) = 1$  for all invariant subset $A$.
\end{enumerate}
\end{definition}
Thus, the ergodicity of an invariant measure is defined in terms of invariant sets. However, we can alternatively define it by backward invariant sets. This observation will be important for generalizing the notion of ergodicity to situations without uniqueness.

\begin{lemma}\label{lem_bkwd}
Let $\Phi$ be a flow on $X$ and $\mu$ be an invariant measure. Then, the following are equivalent.

\begin{enumerate}
 \item $\mu$ is ergodic.
 \item For all backward invariant subset $A$, $\mu(A) = 1$ if $\mu(A) >0$.
\end{enumerate}
\end{lemma}
\begin{proof}
Let $\mu$ be an ergodic invariant measure and $A$ be a backward invariant set with $\mu(A) > 0$. Then, it can be checked that
\[
 B = \bigcap_{t>0} \qty( \Phi^t)^{-1} A
\]
is an invariant set with $\mu(B) = \mu(A)$. Therefore, $\mu(A) = 1$.

Since an invariant set is backward invariant, (2) implies (1). 
\end{proof}
\subsection{Definition and basic properties of generalized dynamical systems}

Now, we review the outline of Yorke's theory of the axiomatic approach to dynamical systems defined by differential equations. Details can be found in \cite{suda2022equivalence, suda2023equivalence, SUDA20241}.

A key observation is the existence of a natural flow on the space of all possible orbits:

\begin{theorem}
 The shift map $\sigma: \mathbb{R} \times C(\mathbb{R}, X) \to C(\mathbb{R}, X) $, defined by 
 \[
  \sigma(t,\phi)(s) := \phi(s+t),
 \]
 is a continuous flow.
\end{theorem}
Thus, any shift-invariant subset $S \subset C(\mathbb{R},X)$ defines a flow $(\sigma, S)$.
Also we define
\[
 \pi: S \to X
\]
by evaluation at 0.
\begin{definition}[Yorke's axioms \cite{yorke1969spaces, suda2023equivalence}]
Let $S \subset C(\mathbb{R},X)$ be shift-invariant and $\pi: S \to \ X$ be the evaluation-at-0 map.
\begin{enumerate}
 \item $S$ satisfies the \textbf{compactness axiom} if $S$ is compact.
 \item $S$ satisfies the \textbf{existence axiom}  if $\pi$ is surjective.
 \item $S$ satisfies the \textbf{uniqueness axiom} if $\pi$ is injective. 
 \item $S$ satisfies the \textbf{switching axiom} if, for each $\phi$ and $\psi$ in $S$ with $\phi(t) = \psi(t')$, the concatenated orbit
 \[
  \psi\cdot \phi(s) = \begin{cases}
      \phi(s) & s \leq t\\
      \psi(s-t+t') & s \geq t
     \end{cases}
 \]
 is contained in $S$.
 \end{enumerate}
\end{definition}
\begin{example}
 Solutions of the half-disk system define a shift-invariant subset $S \subset C(\mathbb{R}, D)$. It satisfies the existence axiom, but not the others. Its closure $\bar S$ satisfies the compactness axiom as well.
\end{example}
The fundamental theorem of Yorke's axiomatic theory is as follows. It implies that we can describe general dynamical systems by dropping some of the axioms listed above.
\begin{theorem}
A shift-invariant subset $S \subset C(\mathbb{R}, X)$ satisfies the existence, uniqueness, and compactness axioms if and only if there is a flow $\Phi: \mathbb{R} \times X \to X$ such that
\[
 S = S_\Phi := \{\Phi(-,x) \mid x \in X\}.
\]
In this case, $(\sigma, S)$ and $( \Phi, X)$ are topologically conjugate.
\end{theorem}
\begin{corollary}
If $\Phi: \mathbb{R} \times X \to X$ is a flow, then $S_\Phi$ is homeomorphic to $X$.
\end{corollary}

Now, we turn to the description of invariant sets. As we cannot assume uniqueness of trajectories, this notion comes with weak and strong versions.
\begin{definition}
Let $S \subset C(\mathbb{R}, X)$ be a shift-invariant subset. A subset $A \subset X$ is \textbf{forward weakly invariant} if, for all $x\in A$, there exists $\phi \in S$ such that $\phi(0) = x$ and $\phi(t)\in A$ for all $t\geq 0$. A subset $A \subset X$ is \textbf{backward strongly invariant} if, for all $x\in A$ and $\phi \in S$ with $\phi(0) = x$, we have $\phi(t)\in A$ for all $t\leq 0$.
\end{definition}
Weak and strong invariant sets can be characterized in terms of shift invariant subsets of $S$:
\begin{theorem}\label{thm_inv}
Let $S \subset C(\mathbb{R}, X)$ be a shift-invariant subset and $A \subset X$.
\begin{enumerate}
 \item The subset $A$ is weakly forward (backward) invariant if and only if there exists a forward (backward) invariant subset $W \subset S$ such that $A = \pi W$.
 \item The subset $A$ is strongly forward (backward) invariant if and only if $\pi^{-1} A$ is forward (backward) invariant.
\end{enumerate}
\end{theorem}
\begin{proof}
These results follow from Theorem 3.18 and Theorem 3.19 in \cite{SUDA20241}.
\end{proof}
\begin{lemma}
Subset $A$ is forward strongly invariant if $X\backslash A$ is backward strongly invariant.
\end{lemma}
\begin{proof}
Let $x \in A$ and assume that there exists $\phi \in S$ with $\phi(0) = x$ that escapes $A$ at some time $t >0$. By backward invariance, $\phi(t) \not \in A$ implies $x \not \in A$, which is a contradiction. 
\end{proof}
The following construction gives the set-valued description of the dynamics.
\begin{definition}
Let $S \subset C(\mathbb{R}, X)$ be a shift-invariant subset. For a subset $E \subset X$, we define
\[
 V_S(t) E := \{\phi(t) \mid  \phi \in S \text{ s.t. } \phi(0) \in E\}.
\]
Also, we define
\[
 V_S(t)^{-1} E := \{x \in X \mid \text{ there exists } \phi \in S \text{ s.t. } \phi(0) = x \text{ and } \phi(t) \in E\}.
\]
\end{definition}
\begin{lemma}
If subset $A$ is forward weakly invariant, then $A \subset V_S(t)^{-1} A $ for all $t \geq 0$. 
\end{lemma}
\begin{proof}
Let $x \in A$. By the forward weak invariance, there exists $\phi \in S$ with $\phi(t) \in A$ for all $t \geq 0$. Therefore $x \in V_S(t)^{-1} A $ for all $t \geq 0$. 
\end{proof}
\begin{lemma}\label{lem_bkwdsubset}
A subset $A$ is backward strongly invariant if and only if $V_S(t)^{-1} A \subset A$ for all $t \geq 0$. \end{lemma}
\begin{proof}
Let $x \in V_S(t)^{-1} A $. Then, there exists $\phi \in S$ with $\phi(t) \in A$ and $\phi(0)= x$. By the backward strong invariance, $x \in A$. Conversely, let $V_S(t)^{-1} A \subset A$ for all $t \geq 0$. For each $x \in A$ and $\phi \in S$ with $\phi(0) = x$, we have $\phi(-s) \in V(s)^{-1} A \subset A$ for all $s \geq 0$. Thus, $A$ is strongly backward invariant.
\end{proof}
Invariant measures can be defined as follows:
\begin{definition}[Invariant measure \cite{SUDA20241}]
Let $S$ be a compact shift-invariant subset of $C(\mathbb{R}, X)$.
A Borel probability measure $\mu$ on $X$ is \textbf{invariant} if there is a shift-invariant measure $\nu$ on $S$ such that $\mu = \pi_*\nu$.
\end{definition}
One of the basic properties of the invariant measures is that they do not decrease along backward transformation via $V_S$:
\begin{lemma}
Let $\mu$ be an invariant measure of a shift-invariant subset $S \subset C(\mathbb{R}, X)$. For all $t \in \mathbb{R}$ and Borel set $A \subset X$, we have 
\[\mu(A) \leq \mu(V_S(t)^{-1} A).\]
\end{lemma}
\begin{remark}
Supports of invariant measures are weakly invariant. Indeed, $\mathrm{supp} \,\mu = \pi\, \mathrm{supp} \,\nu$ by definition, so $\mathrm{supp} \,\mu$ is weakly invariant by Theorem \ref{thm_inv}.
\end{remark}

\section{Definition and basic property of ergodicity}\label{section_ergodic}
Now we consider the definition of ergodicity for dynamical systems without uniqueness of orbits and examine its basic properties.

In what follows, we assume $S \subset C(\mathbb{R}, X)$ to be a compact shift-invariant subset. Other assumptions will be indicated if necessary.

A basic idea of ergodicity is that every invariant set is either very large or small if measured by an invariant measure. 
While every notion of invariance can define its own version of ergodicity, it leads to triviality if it is too strong or too weak. The next result asserts that the backward strongly invariant subsets are invariant in the almost everywhere sense. This observation suggests that it is a natural generalization of the notion of invariant sets when trajectories are not uniquely determined.
\begin{lemma}\label{lem_bkwdsame}
Let $A \subset X$ be a backward strongly invariant subset and $\mu$ be an invariant measure. Then,
\[
 \mu(V_S(t)^{-1} A \,\Delta\, A )=0
\]
for all $t \geq 0$. In particular, we have
\[
 \mu(V_S(t)^{-1} A ) = \mu(A).
\]
\end{lemma}
\begin{proof}
By Lemma \ref{lem_bkwdsubset}, $V_S(t)^{-1} A \subset A$, so we have to check that \[\mu\qty(A \backslash V_S(t)^{-1} A ) = 0.\] This equality follows from $\mu(A) \leq \mu\qty( V_S(t)^{-1} A )$, which results from invariance.
\end{proof}
Given this observation, we consider the following formulation.
\begin{definition}\label{def_ergodicity}
Let $\mu$ be an invariant measure of $S$ with $\mu = (\pi)_* \nu$, where $\nu$ is a shift invariant measure. 
\begin{enumerate}
 \item $\mu$ is \textbf{ergodic} if $\mu(A) = 0$ or $1$ for all backward strongly invariant set $A$.
 \item $\mu$ is \textbf{finely ergodic} if $\nu$ is ergodic.
\end{enumerate}
\end{definition}
The next result clarifies the relation with the invariant measure on the shift-invariant space.
\begin{lemma}
Let $\mu$ be an ergodic invariant measure with a shift invariant measure $\nu$ such that $\mu = (\pi)_*\nu$. 
\begin{enumerate}
 \item If $\mu$ is finely ergodic, $\mu$ is ergodic.
 \item If $S$ satisfies the uniqueness axiom,  every ergodic $\mu$ is finely ergodic.
\end{enumerate}
\end{lemma}
\begin{proof}
(1) Let $A$ be a backward strongly invariant subset with $\mu(A) > 0$. Then, $\pi^{-1}(A)$ is backward invariant with respect to the shift map. As $\mu(A) = \nu\qty(\pi^{-1}(A))$, we have $\mu(A) = 1$ by Lemma \ref{lem_bkwd}.

(2) Let $W \subset S$ be shift invariant. By the uniqueness, we have $W = \pi^{-1} \circ \pi (W).$ Therefore, $\nu(W) = \nu\qty(\pi^{-1} \circ \pi (W)) = \mu(\pi (W)).$ Since $\pi(W)$ is  strongly backward invariant, $\nu(W) = 0$ or $1$.
\end{proof}
\begin{remark}
If $S$ is compact, there is at least one finely ergodic invariant measure.
\end{remark}
\begin{remark}
For an invariant measure $\mu$, being finely ergodic does not imply that every weakly backward invariant set has measure either 0 or 1. The converse is not true. 
\end{remark}
The following observation is often useful to establish the existence of ergodic measures in concrete examples.
\begin{lemma}
If $X$ is the only strongly backward invariant set of $S$, then every invariant measure is ergodic.
\end{lemma}
\begin{example}[Bean model]
The bean model considered in \cite{buzzi2016chaotic}  is one of the simplest examples of piecewise-smooth dynamical systems with chaotic behavior. This is a piecewise-smooth planar vector field defined on
\[
 \Lambda := \{(x,y) : -1 \leq x \leq 1, (x^4-x^2)/2 \leq y \leq 1-x^2\}
\]
and it is given by
\[
 X(x,y) = \begin{cases}
    (1,-2 x) & y > 0\\
    (-2, 4x^3 - 2 x) & y <0.
   \end{cases}
\] 

It is straightforward to show that the only strongly backward invariant set is the whole phase space. Thus, the bean model is ergodic with respect to every invariant measure.
\end{example}
\begin{example}
The half-disk system is ergodic. Indeed, we observe that the only strongly backward invariant set is the whole phase space. 
\end{example}
Under this definition of ergodicity, results on the non-decomposability hold analogously to the classical case. First, extremal invariant measures are ergodic.
\begin{theorem}
Extremal invariant measures are finely ergodic.
\end{theorem}
\begin{proof}
Let $\mu$ be an extremal invariant measure of $S$ with $\mu = \pi_* \nu$, where $\nu$ is a shift invariant measure. 
It suffices to show that  $\nu$ is extremal because extremal invariant measures are ergodic in the classical setting. Let $\nu = \alpha \nu_1 + (1-\alpha)\nu_2$, where $\alpha \in [0,1]$ and $\nu_1, \nu_2$ are shift invariant. Then, by applying $\pi_*$, it is obvious that $\nu$ is extremal as well. 
\end{proof}
A useful characterization of ergodicity is that invariant functions are constant almost everywhere. The following results generalize this property.
\begin{theorem}
Let $\mu$ be an invariant measure of $S$. Then, the following are equivalent.
\begin{enumerate}
 \item $\mu$ is ergodic.
 \item A measurable $\{0,1\}$-valued function is $\mu$-almost everywhere constant if it is non-increasing along trajectories in $S$.
\end{enumerate}
\end{theorem}
\begin{proof}
(1) $\Rightarrow$ (2): Let $f:X \to \{0,1\}$ be measurable and non-increasing along trajectories in $S$. Then, $A :=\{x \mid f(x) = 1\}$ is strongly backward invariant and therefore $f$ is constant $\mu$-almost everywhere.

(2) $\Rightarrow$ (1): Let $A$ be strongly backward invariant. Then, $\chi_A$ is non-increasing along trajectories in $S$ and therefore constant $\mu$-almost everywhere. This implies either $\mu(A) =0 $ or $\mu(A) =1 $.
\end{proof}
The next result establishes the topological transitivity on the support of the ergodic measure.
\begin{theorem}
If $S$ satisfies the switching axiom and $\mu$ is ergodic, $\mu$-almost every point $x$ has a forward orbit dense in the support of $\mu$. 
\end{theorem}
\begin{proof}
Let $U$ be an open subset of the support of $\mu$. Then, $W = \cup_{t \geq 0} V_S(t)^{-1} U$ is backward strongly invariant by the switching axiom. Therefore, $\mu(W) = 1$ and almost every point has a forward trajectory visiting $U$. 

Let $\{U_n\}_n$ be a countable basis of the topology of the support of $\mu$. We define \[A_n := \{x \mid \exists \phi \in S \text{ such that } \phi(t_j) \in U_j \text{ for } 1 \leq j \leq n \text{ with } t_j > 0\}.\] 
By the construction above, $\mu(A_1) = 1$.
Let $W_2 =  \cup_{t \geq 0} V_S(t)^{-1} U_2$. As $\mu(W_2 \cap U_1) = \mu(U_1) > 0$, $\cup_{t \geq 0} V_S(t)^{-1}\qty(W_2 \cap U_1)$ is a backward strongly invariant set with positive measure. This has full measure, and consequently almost every point has a trajectory that first visits $U_1$ and then $U_2$. Therefore,  $\mu(A_2) = 1$. 

Similarly, the set
\[
 \bigcup_{t_1 \geq 0}V(t_{1})^{-1}\qty( \bigcup_{t_2 \geq 0} V(t_{2})^{-1}\qty( \cdots \bigcup_{t_{n-1} \geq 0} V(t_{n-1})^{-1} (W_n \cap U_{n-1}) \cdots \cap U_2 )\cap U_1)
\]
 has full measure for all $n$. Therefore $\mu(A_n) = 1$ for all $n$. Since $A_{n+1} \subset A_n$, $\mu(\cap_n A_n) =1$. This implies that $\mu$-almost every point $x$ has a forward orbit dense in the support of $\mu$.
\end{proof}
One way to find an ergodic measure is to consider weakly invariant subsets with a flow residing on them.
\begin{definition}
Let $S$ be a shift-invariant subset. For a weakly invariant subset $A \subset X$, a flow $\Phi$ on $A$ is a \textbf{subflow} of $S$ if $\{\Phi_x \mid x \in A\} \subset S$.
\end{definition}
\begin{lemma}
Let $S$ be a compact shift-invariant subset with switching axiom and $A \subset X$ be a weakly invariant subset of $X$ with a subflow $\Phi$. If $\mu_0$ is an invariant measure of $\Phi$, $\mu := (j_A)_*\mu_0$ defines an invariant measure, where $j_A: A \to X$ is the inclusion.
 \end{lemma}
 \begin{proof}
 In this case, we can use the characterization of invariant measure by the non-decreasing property by Theorem 5.12 in \cite{SUDA20241}. Let $B$ be a Borel subset. For all $t \geq 0$, we have $\qty(\Phi^t)^{-1} (B\cap A) \subset V_S(t)^{-1} B$. Therefore,
 \[
 \begin{aligned}
  \mu\qty( V_S(t)^{-1} B ) &\geq \mu\qty( \qty(\Phi^t)^{-1} (B\cap A) ) \\
 &= \mu_0\qty( \qty(\Phi^t)^{-1} (B\cap A) )\\
 & =  \mu_0\qty(B\cap A ) = \mu(B)
 \end{aligned}
 \]
 for all $t \geq 0$.
 \end{proof}
 \begin{theorem}
Let $S$ be a compact shift-invariant subset with switching axiom and $A \subset X$ be a weakly invariant subset of $X$ with a subflow $\Phi$. If $\mu_0$ is an ergodic invariant measure of $\Phi$, then $\mu := (1_A)_*\mu_0$ is ergodic.
 \end{theorem}
 \begin{proof}
 Let $B$ be a backward strongly invariant set with $\mu(B) > 0$. Then, $A \cap B$ is backward $\Phi$ invariant. Therefore we have
 \[
  \mu(B) = \mu_0(A\cap B) = 1
 \]
 by ergodicity of $\mu_0$.
 \end{proof}
 \begin{example}
 The half-disk system has an ergodic invariant measure with support on the boundary of the domain. Obviously, it has a subflow on the boundary of the domain consisting only of a periodic orbit. Thus, there is an ergodic invariant measure with support on that orbit.
 \end{example}
\section{ Birkhoff ergodic theorem and mixing properties}\label{section_birkhoff}
In this section, we consider consequences of the Birkhoff ergodic theorem and mixing properties of ergodic systems. Since generalized dynamical systems are shift-invariant subsets of $C(\mathbb{R}, X)$, we can apply the classical Birkhoff ergodic theorem and consider its implications.

First, a direct application of the Birkhoff ergodic theorem yields:
\begin{theorem}
Let $S \subset C(\mathbb{R}, X)$ be a compact shift-invariant subset. Let $\mu$ be an invariant measure with a shift invariant measure $\nu$ such that $\mu = (\pi)_*\nu$. Then, for $\nu$-a.e. $\phi \in S$ and every $f \in L^1(\mu)$,
\[
 \bar f(\phi) := \lim_{\tau \to \infty} \frac{1}{\tau} \int_0^\tau f(\phi(t)) \mathrm{d} t
\]
exists. Further, we have
\[
 \int_S \bar f(\phi) \nu(\mathrm{d} \phi) = \int_X f(x) \mu(\mathrm{d} x). 
\]

\end{theorem}

In general, we cannot expect $\bar{f}(\phi)$ to lead to a well-defined function on $X$ due to non-uniqueness of the trajectories. Therefore, it will be of interest to consider the set of all possible values of $\bar{f}$ when we start from a given point $x$ and choose $\phi$.

\begin{proposition}\label{prop_rare}
Let $S \subset C(\mathbb{R}, X)$ be a compact shift-invariant subset satisfying the switching axiom. Let $\mu$ be an invariant measure with a shift invariant measure $\nu$ such that $\mu = (\pi)_*\nu$. Then, for all $f \in L^1(\mu)$ and $c \in \mathbb{R}$, the set
\[
 \pi \qty(\bar{f}^{-1}(c)) := \{x \in X \mid \text{ there exists } \phi \in S \text{ such that } \bar{f}(\phi) = c\}
\]
is backward strongly invariant. In particular, if $\mu$ is ergodic, this has measure either 0 or 1.
\end{proposition}
\begin{proof}
Let $x\in \pi \qty(\bar{f}^{-1}(c))$ and $\psi \in S$ with $\psi(0) = x$. By assumption, there exists $\phi \in S$ with $\phi(0) = x$ such that $\bar{f}(\phi) = c$. Then, the concatenation $\phi \cdot \psi$ of $\phi$ after $\psi$ is contained in $S$ by the switching axiom and we have $\bar{f}(\phi \cdot \psi) = c$. Therefore, we have $\phi(t) \in \pi \qty(\bar{f}^{-1}(c))$ for all $t < 0$.
\end{proof}

For finely ergodic invariant measures, we may apply the Birkhoff ergodic theorem to show the validity of the empirical measure. This provides us with a way to obtain information about the system in terms of individual trajectories. 
\begin{corollary}\label{cor_birkhoff}
Let $S \subset C(\mathbb{R}, X)$ be a compact shift-invariant subset. Let $\mu$ be a finely ergodic invariant measure with an ergodic shift invariant measure $\nu$ such that $\mu = (\pi)_*\nu$. Then, for $\nu$-almost every $\phi \in S$ and measurable $A \subset X$, we have
\[
 \lim_{t \to \infty} \frac{1}{t}\int_0^t \chi_A(\phi(s)) \mathrm{d} s = \mu(A).
\] 
\end{corollary}

Further, we may obtain a weak Birkhoff-type result in terms of the set-valued semigroup $V_S(t)$, removing the reference to the choice of trajectories.
\begin{theorem}
Let $S \subset C(\mathbb{R}, X)$ be a compact shift-invariant subset. Let $\mu$ be a finely ergodic invariant measure with an ergodic shift invariant measure $\nu$ such that $\mu = (\pi)_*\nu$. Then, for all measurable subsets $A$ and $B$ of $X$,
\[
 \liminf_{t\to \infty} \frac{1}{t} \int_{0}^t \mu\qty(V_S(s)^{-1}A \cap B) \mathrm{d}s \geq \mu(A) \mu(B).
\]
\end{theorem}
\begin{proof}
First, we observe $V_S(s)^{-1} A = \pi \circ \sigma^{-s} \circ \pi^{-1} A.$ This implies \[\chi_{V_S(s)^{-1} A} \circ \pi (\phi)\geq \chi_{\sigma^{-s} \circ \pi^{-1} A} (\phi) = \chi_{ A} (\phi(s))\] for all $\phi \in S$. By integrating both sides of the inequality and applying Corollary \ref{cor_birkhoff}, we have
\[
 \liminf_{t\to \infty} \frac{1}{t}\int_0^t \chi_{V_S(s)^{-1} A} \circ \pi (\phi) \mathrm{d} s \geq \mu(A).
\]
for $\nu$-almost every $\phi$. Therefore, we have
\[
 \liminf_{t\to \infty} \frac{1}{t}\int_0^t \chi_{V_S(s)^{-1} A}  (x) \mathrm{d} s \geq \mu(A).
\]
for $\mu$-almost every $x \in X$.
By Tonelli's theorem and Fatou's lemma, we have
\[
\begin{aligned}
 \liminf_{t\to \infty} \frac{1}{t} \int_{0}^t \mu\qty(V_S(s)^{-1}A \cap B) \mathrm{d}s &= \liminf_{t\to \infty} \frac{1}{t} \int_{0}^t \int \chi_{V_S(s)^{-1}A }(x) \chi_B(x) \mathrm{d} \mu (x)\mathrm{d}s\\
  &= \liminf_{t\to \infty}  \int  \qty(\frac{1}{t} \int_{0}^t   \chi_{V_S(s)^{-1}A }(x)  \mathrm{d}s  )\chi_B(x) \mathrm{d} \mu (x)\\
  &\geq \int  \qty( \liminf_{t\to \infty} \frac{1}{t} \int_{0}^t   \chi_{V_S(s)^{-1}A }(x)  \mathrm{d}s  )\chi_B(x) \mathrm{d} \mu (x)\\
  &\geq \mu(A) \mu(B).
 \end{aligned}
\]
as desired.
\end{proof}
The inequality cannot be improved to an equality, as the next example shows.
\begin{example}
We consider the compact shift-invariant set $\bar{S}$ derived from the half-disk system.  Let $\mu$ be the invariant measure induced by the Dirac measure with support at the constant function $(0,0)$. Let $A$ be a measurable subset such that $(0,0) \not\in A$. It is easy to observe that we have $\mu(A) = 0$ but $\mu\qty(V_S(t)^{-1}A) = 1$ for sufficiently large $t$. Therefore, by taking $B = \{(0,0) \}$, we have
\[
  \liminf_{t\to \infty} \frac{1}{t} \int_{0}^t \mu\qty(V_S(s)^{-1}A \cap B) \mathrm{d}s = 1> 0 = \mu(A) \mu(B) .
\]
\end{example}
Finally, we note that the mixing properties can also be generalized.
\begin{definition}\label{def_wmix}
An invariant measure $\mu$ is \textbf{weakly mixing} if 
\[
 \lim_{t \to \infty} \frac{1}{t}\int_0^t \left|\mu\qty(V_S(t)^{-1} A \cap B) - \mu(A) \mu(B) \right| = 0
\]
for all Borel subsets $A$ and $B$. 
\end{definition}
\begin{theorem}\label{thm_wmix}
If $\mu$ is weakly mixing, then $\mu$ is ergodic.
\end{theorem}
\begin{proof}
Let $A$ be a backward strongly invariant subset. Then, we have
\[
 \mu\qty(V_S(t)^{-1} A \cap A) = \mu\qty(V_S(t)^{-1} A ) = \mu\qty(A) 
\]
by Lemmas \ref{lem_bkwdsubset} and \ref{lem_bkwdsame}. The rest of the proof proceeds in the same way as the case of a flow.
\end{proof}
\section{Concluding Remarks}\label{section_conclusion}
For an ergodic system, a possible value of the time average is either very common or very rare by Proposition \ref{prop_rare}. A natural question here is: \textit{Given $f \in L^1(\mu)$, where $\mu$ is ergodic, is it possible to determine which value can be realized as the time average almost everywhere?}

Another important question is the relation to the categorical definition of ergodicity considered in \cite{moss2023category}. One can formulate the notion of ergodicity for systems in the Markov category, and therefore we can expect to obtain a version of ergodicity by applying this definition. However, to make this possible in the current setting, we need to formulate a suitable category, which is an interesting question in itself.
\bibliography{ergodic}
\bibliographystyle{plain}

\end{document}